\documentclass[11pt, reqno]{amsart}
\usepackage{geometry,hyperref,enumitem,cancel,lineno}
\usepackage{changes}
\usepackage[utf8]{inputenc}
\usepackage{xcolor}
\usepackage{graphicx}
\usepackage{epstopdf}
\usepackage[mathscr]{euscript}
\usepackage{amsrefs}
\usepackage{amsmath}
\usepackage{amsthm}
\usepackage{amssymb}
\usepackage{amsfonts}
\usepackage{changes}
\usepackage{ulem}
\usepackage{faktor}

\numberwithin{equation}{section}
\theoremstyle{plain}
\newtheorem{theo}{Theorem}[section]
\newtheorem{lem}[theo]{Lemma}
\newtheorem{prop}[theo]{Proposition}

\newtheorem{definition}[theo]{Definition}
\newtheorem{defi}[theo]{Definition}

\newcommand{\I}{\mathrm{i}}
\newcommand{\ks}{\operatorname{k}}
\renewcommand{\Re}{\operatorname{Re}}
\newcommand{\p}{\partial}
\newcommand{\Ric}{\operatorname{Ric}}
\newcommand{\cI}{\mathscr{I}}

\newcommand{\cZ}{\mathscr{Z}}
\newcommand{\cP}{\mathscr{P}}
\newcommand{\cW}{\mathscr{W}}
\newcommand{\cH}{\mathscr{H}}
\newcommand{\cQ}{\mathscr{Q}}
\newcommand{\cT}{\mathscr{T6}}

\author{Masoud Ganji}
\address{M.G. University of New England, School of Science and Technology, Armidale NSW 2351, Australia}
\email{mganjia2@une.edu.au}
\author{Gerd Schmalz}
\address{G.S. University of New England, School of Science and Technology, Armidale NSW 2351, Australia}
\email{schmalz@une.edu.au}
\author{Daniel Sykes}
\address{D.S. University of New England, School of Science and Technology, Armidale NSW 2351, Australia}
\email{dsykes5@une.edu.au}

\title[Quasi-Einstein shearfree spacetimes lifted from Sasakian manifolds]{Quasi-Einstein shearfree spacetimes lifted from Sasakian manifolds}

\date{}

\begin{document}
 \rightline{\today}
\definechangesauthor[name={GS}, color=red]{GS}
\definechangesauthor[name={D}, color=blue]{D}
\definechangesauthor[name={MG}, color=purple]{MG}
\maketitle

\begin{abstract}
   In this article we prove that a certain class of {\it smooth} Sasakian manifolds admits lifts to 4-dimensional quasi-Einstein shearfree spacetimes of Petrov type II or D. This is related to an analogous result by Hill, Lewandowski and Nurowski \cite{HLN} for general {\it real-analytic} CR manifolds. In particular, our result holds for smooth tubular CR manifolds. Furthermore, we show that any Sasakian manifold with underlying Kähler-Einstein manifold with non-zero Einstein constant has a lift to a shearfree Einstein metric of Petrov type II or D.
\end{abstract}


\section{Introduction}
		

A foliation by integral curves of a non-vanishing null vector field $\ks$ on a Lorentzian manifold ($\mathscr{M},g)$ is called a shearfree congruence if
$$\mathcal L_{\ks} g(v,w) = \rho g(v,w),$$
for some function $\rho$ and for all vector fields $v,w \in \ks^\perp$. These integral curves are automatically geodesics and in the case of a 4-dimensional Lorentzian manifold they can be interpreted as light rays. The 4-dimensional case directly applies to General Relativity and has been intensively studied by many authors, see e.g. \cites{DKS, RT1, LNT, Arman1, Arman2, Arman3, Rob, Traut1, Tafel1, RT1, Traut2, Traut3, Traut4, G, XF1, XF2}. The famous Goldberg-Sachs Theorem \cite{GoldSa} (See also  \cite{GHN} for a generalisation of the Goldberg-Sachs theorem.) relates the existence of shearfree congruences to algebraically special solutions of the Einstein field equations for which the Weyl tensor satisfies certain degeneracy conditions, expressed through its so-called Petrov type.  Of particular interest are the shearfree congruences with underlying Lorentzian metrics $g$ that satisfy the Einstein condition
$$\Ric (g) = \Lambda g,$$
or the relaxed quasi-Einstein condition
\begin{equation}\label{qec}
    \Ric (g) = \Lambda g + \Phi \lambda^2,
\end{equation}
where $\Ric (g)$ is the Ricci curvature of $g$, $\Lambda$ is the so-called cosmological constant, $\lambda = g(\ks,\cdot)$ and $\Phi$ is a smooth function corresponding to the energy momentum tensor of pure radiation.

Robinson and Trautman \cite{RT1} described the relation between shearfree congruences in 4-dimensional spacetimes and the CR structure of their 3-dimensional leaf spaces. A CR structure on a 3-dimensional manifold is a distribution $D$ of rank 2 equipped with a field of endomorphisms $J\colon D \to D$ with $J^2 = \operatorname{-Id}$. The distribution $D$ on the leaf space $M$ is induced by the distribution $\ks^\perp$ on $\mathscr{M}$, and $J$ is induced by the conformal structure on $\faktor{\ks^\perp}{\ks}$. The CR structure of $M$ can be encoded as a pair $(\mu,\lambda)$ comprised of a complex 1-form $\mu$ and a real 1-form $\lambda$, such that $D = \ker \lambda$ and $\mu$ is a complex coordinate on $D$, that is $\mu\wedge\bar{\mu}\wedge \lambda\neq 0$. By $[(\mu,\lambda)]$ we denote the equivalence class of pairs $(\mu,\lambda)$ that produce the same CR structure on $M$.

Locally, any shearfree Lorentzian metric associated with a given CR manifold $(M,[(\mu,\lambda)])$ can be written as
\begin{equation*}
    g = 2{\cP}^2\left(\mu \bar{\mu} + \lambda\left(dr+\cW\mu + \overline{\cW}\bar{\mu} + \cH \lambda\right)\right),
\end{equation*}
where $\cP\neq 0, \cH, r$ are real functions, and $\cW$ is a complex function on $\mathscr M$ (see e.g. \cites{RT1,HLN}). Any metric from this family is called a lift of the underlying CR structure. 

As one of the main results in the paper \cite{HLN}, Hill, Lewandowski and Nurowski proved that for any real-analytic CR manifold, there exists a lift that renders the Lorentzian manifold $(\mathscr{M},g)$ {\it quasi-Einstein}, that is, \eqref{qec} is satisfied. Moreover, $g$ is of Petrov type II or D. 

In this paper, we focus on the case when the underlying CR manifold has a Sasakian structure, i.e., it is strictly pseudoconvex and it admits an infinitesimal CR automorphism transversal to the distinguished distribution $D$ (see Section 2 for precise definitions).  The main result of this article is Theorem \ref{sasaki-theoint}, which is an analog to the theorem by Hill, Lewandowski and Nurowski, cited above. We remove the assumption of real analyticity and prove that a particular class of \emph{smooth} 3-dimensional Sasakian manifolds admit lifts to quasi-Einstein Lorentzian manifolds with positive cosmological constant $\Lambda$ and of Petrov type II or D. Petrov type II or D means that the Weyl tensor is to some degree degenerate, that is certain components vanish and another certain component is different from $0$. In what follows, we use $(z = x + \I y, w = u + \I v)$ as the local coordinates of $\mathbb{C}^{2}$.
\begin{theo}\label{sasaki-theoint}
Let $\left(M,D,J\right)$ be a 3-dimensional Sasakian manifold given by a defining equation $v = F(z,\bar{z})$, where $F\colon \mathbb C \to \mathbb R$ is such that the limits
    \begin{equation}
        \displaystyle{\lim_{|z| \to \infty}}F_{z\bar{z}} > 0 \quad \text{and} \quad \displaystyle{\lim_{|z| \to \infty}}\partial_{z\bar{z}}\log \left(F_{z\bar{z}}\right) < 0
    \end{equation}
exist, and let $\Lambda$ be an arbitrary positive constant. Then, there exists a 4-dimensional Lorentzian manifold $(\mathscr M = M \times \mathbb{R}, g)$, which is a lift of $M$, such that $g$ is of Petrov type II or D, and satisfies the quasi-Einstein equation
$$\Ric (g) = \Lambda g + \Phi \lambda^2,$$
where $\Phi$ is some real function.
\end{theo}
We illustrate this result in the special case of tubular hypersurfaces. In this case the cosmological constant can be arbitrary.
\begin{theo}\label{tubular-theo}
Let $\left(M,D, J\right)$ be a 3-dimensional tubular CR manifold with defining function $v = F(y)$, and let $\Lambda$ be an arbitrary constant. Then, there exists a 4-dimensional Lorentzian manifold $(\mathscr M = M \times \mathbb R, g)$, which is a lift of $M$, such that $g$ is of Petrov type II or D, and satisfies the quasi-Einstein equation
$$\Ric (g) = \Lambda g + \Phi \lambda^2,$$
where $\Phi$ is some real function.
\end{theo}
In Theorem \ref{ESAKint}, we consider Sasakian manifolds for which the associated Kähler manifold of integral curves of the Reeb vector field is Einstein with non-zero Einstein constant. We show that such Sasakian manifolds admit a lift to an Einstein Lorentzian manifold. 
\begin{theo}\label{ESAKint}
Let $(M,D,J)$ be a 3-dimensional Sasakian manifold with defining function $v = F(z,\bar{z})$, and assume the underlying Kähler manifold is Einstein with non-zero Einstein constant $\Lambda_o$ and let the cosmological constant, $\Lambda$, be any real value such that $\Lambda\Lambda_o > 0$. Then, there exists a 4-dimensional Lorentzian manifold $(\mathscr M = M \times \mathbb R, g)$, which is a lift of $M$, such that $g$ is of Petrov type II or D, and satisfies the Einstein equation
$$\Ric (g) = \Lambda g.$$
\end{theo}
\section{CR manifolds and Sasakian manifolds}
In this section, we recall some facts regarding CR manifolds and Sasakian manifolds, and fix our notational conventions.

\begin{definition}
A 3-dimensional CR manifold is a triple $(M,D,J)$ where $M$ is a 3-dimensional manifold, $D$ a codimension 1 distribution and $J$ a smooth field of endomorphisms with $J^2=-\operatorname{Id}$ on $D$.  
\end{definition}

We assume that $(M,D,J)$ is strictly pseudoconvex, i.e. for any (local) non-vanishing section $X$ of $D$, $[X,JX] \not\in D$ at any point. For any choice of $X$ we have an adapted complex frame $(\partial,\bar{\partial},\partial_o)$, where
$$\partial= X - \I JX, \qquad \partial_o = \I [\partial, \bar{\partial}] = -2[X,JX].$$
The complex vector field $\partial = X - \I JX$ spans the $ +\I$-eigen-distribution $D^{1,0}$ of $J$ in $D \otimes \mathbb C$. We denote the corresponding dual coframe by $(\mu,\bar{\mu},\lambda)$. Since $\bar{\partial}$ and $\bar{\mu}$ are just the conjugates of $\partial$ and $\mu$, we will also refer to $(\partial, \partial_o)$ and $(\mu,\lambda)$ as a frame and coframe, respectively. The real 1-form $\lambda$ and the complex 1-form $\mu$ completely determine the CR structure on $M$, with $D = \ker \lambda$ and $\mu|_D \colon D \to \mathbb C$ as a complex coordinate on $D$. Strict pseudoconvexity of $M$ translates to
$$d\lambda \wedge \lambda \neq 0.$$ 
Our choice implies the structure equations
\begin{align}\label{2.1}
\begin{split}
    d\lambda &= \mathrm{i} \mu \wedge \bar{\mu} + c\mu \wedge \lambda + \bar{c}\bar{\mu} \wedge \lambda\\
    d\mu &= \alpha\mu \wedge \lambda + \beta \bar{\mu} \wedge \lambda,
    \end{split}
\end{align}
%
where $c, \alpha, \beta$ are complex valued functions on $M$. 

Notice that a strictly pseudoconvex CR manifold carries the structure of a contact manifold with contact form $\lambda$. For any choice of a contact form $\lambda$, one defines the uniquely determined \emph{Reeb vector field} $\cZ$ that satisfies $\cZ \, \lrcorner \, \lambda=1$ and $\cZ \, \lrcorner \, d\lambda = 0$.  

Any other adapted frame $(\partial',\bar{\partial}',\partial_o')$ and coframe $(\mu',\bar{\mu}',\lambda')$ express through the original frame and coframe by
\begin{align}
    \partial' &= \frac{1}{f}\partial, \label{trans1}&  
    \qquad \partial_o' &= \frac{1}{|f|^2}\big(\partial_o - \ell\partial - \bar{\ell}\bar{\partial}\big), \\
    \mu' &= f(\mu + \ell \lambda), &\qquad \lambda' &= |f|^2\lambda, \label{trans2}
\end{align}
where $f\ne 0$ and $\ell$ are complex valued functions, and
\begin{align}\label{equa11}
\begin{split}
    &\ell = -\I \bar{\partial}(\log f), \quad \alpha' = \frac{1}{|f|^2}\left(\alpha - \partial_o (\log f) + \ell\partial (\log f) + \partial \ell + \ell c \right), \\
    & c' = \frac{1}{f}\left(c - 2\I\bar{\ell} + \partial (\log f)\right), \quad \beta' = \frac{1}{\bar{f}^2}\left(\beta + \I \ell^2 + \bar{\partial}\ell + \bar{c}\ell\right).
    \end{split}
\end{align}
\begin{definition}
A CR function on a CR manifold $(M,D,J)$ is a complex-valued $\mathcal C^1$-function $\zeta$, such that 
$$\bar{\partial}(\zeta) = 0.$$
This definition does not depend on the choice of the adapted frame.
\end{definition}

A (local) realisation or embedding into $\mathbb C^2$ of a CR manifold $(M,D,J)$ is a (local) $\mathcal C^1$-embedding $M\to \mathbb C^2$ such that both components of the mapping are CR functions. Such embeddability is equivalent to the existence of two functionally independent CR functions $z$ and $w$, i.e.
$$dz \wedge dw \neq 0.$$
If the CR manifold $(M,D,J)$ is (locally) embedded as a real hypersurface
$$v=F(z,\bar{z},u)$$
in $\mathbb C^2$ with local coordinates $(z,w = u + \I v),$ we can choose an adapted frame and coframe as in the Lemma below. 

\begin{lem}\label{lemmaframcoframe}
Let $(M,D,J)$ be a CR manifold  locally embeddable as a real hypersurface in $\mathbb C^2=\{z=x+\I y, w=u+\I v \}$
as $v=F(z,\bar{z},u)$ for some smooth real function $F$. Set
$$L = \frac{F_z}{F_u + \I}.$$
Then there exists a CR coframe $(\mu,\bar{\mu},\lambda)$ and dual frame $(\partial,\bar{\partial},\partial_o)$ where
\begin{align*}
\mu &=dz, \hspace{4cm} \lambda = \frac{du + L\,dz + \bar{L}\,d\bar{z}}{\I \left(\bar{\partial}L - \partial \bar{L}\right)},\\
\partial &= \partial_z - L\partial_u, \hspace{3cm} \p_o = \I \left(\bar{\partial}L - \partial \bar{L}  \right)\p_u.
\end{align*}
The coframe satisfies the structure equations:
$$d\mu=0, \hspace{2cm}d\lambda = \mathrm{i} \mu \wedge \bar{\mu} + c\mu \wedge \lambda + \bar{c}\bar{\mu} \wedge \lambda,$$
where
$$c =  - \p\log \left(\I \left (  \bar{\partial} L- \partial\bar{L} \right ) \right) - L_u.$$
In particular,
$$ \p \bar{c} = \bar{\p} c.$$
\end{lem}
\begin{proof}
Let $\mu = dz,$ and choose a contact form 
\begin{align*}
    \lambda' = J(dv - dF) &= (1 + F^2_u)du + (-\I F_z + F_uF_z)dz + (\I F_{\bar{z}} + F_uF_{\bar{z}})d\bar{z}\\
    &= (1 + F^2_u)(du + L\,dz + \bar{L}\,d\bar{z})
\end{align*}
with $L$ defined as above.
We need to find $\psi$ such that 
$$\lambda = \frac{\psi}{(1 + F_u^2)}\lambda'$$
satisfies \eqref{2.1}, i.e.
$$d\lambda = \psi \, d\left(L\,dz + \bar{L}\,d\bar{z}\right) + d\psi \wedge \frac{1}{\psi}\lambda \equiv \I \,dz \wedge d\bar{z} \ \operatorname{mod} \lambda.$$
Notice that the dual adapted frame to $(\mu,\bar{\mu},\lambda)$ becomes
\begin{equation}\label{coofr}
    \partial = \partial_z - L\partial_u, \quad \bar{\partial} = \partial_{\bar{z}} - \bar{L}\partial_u, \quad \partial_o = \frac{1}{\psi} \partial_u.
\end{equation}
It follows
\begin{equation}\label{coopsi}
    \frac{1}{\psi} = - \I (\partial \bar{L} - \bar{\partial}L) = - \I (\bar{L}_z - L_{\bar{z}} + L_{u}\bar{L} - \bar{L}_{u}L),
\end{equation}
and therefore
\begin{align*}
    \lambda = \frac{du + L\,dz + \bar{L}\,d\bar{z}}{\I \left(\bar{\partial}L - \partial \bar{L}\right)}, \quad  \p_o = \I \left(\bar{\partial}L - \partial \bar{L}  \right)\p_u.
\end{align*}
The structure function $c$, in \eqref{2.1}, now takes the form
\begin{equation}\label{cdef}
    c = \frac{\partial \psi}{\psi} - \psi\partial_{o}L = - \p\log \left(\I \left (  \bar{\partial} L- \partial\bar{L} \right ) \right) - L_u.
\end{equation}
Moreover, it follows from $\mu = dz$ and $d^2\lambda = 0$ that
\begin{equation}\label{cid}
    \p \bar{c} = \bar{\p} c.
\end{equation}
\end{proof}

In the remainder of the paper, we use Cartan's description of a CR manifold $M$, defining the CR structure by the equivalence classes $[(\mu,\lambda)]$ of adapted coframes, defined as above in \eqref{trans2}. The definition of Sasakian manifolds given below suits the purpose of this paper best. For other definitions and their equivalence see \cite{ACHK} or \cite{G}.
\begin{defi}
A Sasakian manifold (of dimension 3) is a strictly pseudoconvex CR manifold $(M,D,J)$ endowed with a vector field $\cZ$ transversal to $D$,
which is an infinitesimal automorphism, i.e. 
    \begin{equation}\label{infcr}
        \mathscr{L}_{\cZ} D\subseteq D,\qquad \mathscr{L}_{\cZ} J = 0.  
    \end{equation}
\end{defi}
Notice that the first condition in \eqref{infcr} is equivalent to $\cZ$ being the Reeb vector field for some contact form $\lambda$. In terms of a chosen adapted frame the second condition in \eqref{infcr} is equivalent to
$$\bar{\partial} \, \lrcorner \, \mathcal L_{\cZ} \mu = 0.$$
It is well known, see e.g. \cites{BRT,Jac,EKS1}, and easy to see that (3-dimensional) Sasakian manifolds are locally realisable as submanifolds of $\mathbb{C}^2$. Indeed, consider the manifold $M \times \mathbb R_r$. The complex vector fields $\partial$ and $\cZ + \I \partial_r$ define an integrable complex structure on $M \times \mathbb R$ in which $M$ is embedded as $r = 0$. In suitable local coordinates $z, w = u + \I v$ on the complex manifold $M \times \mathbb R$, $M$ can be expressed as
\begin{equation}\label{embed}
    v = F(z,\bar{z}).
\end{equation}
The leaf space of the integral curves of the Reeb vector field $\cZ=\partial_u $ is in this case simply $\mathbb C_z$. It has a natural Kähler structure with Kähler potential $F(z,\bar{z})$ and Kähler metric $h = 2F_{z\bar{z}}dz \, d\bar{z}$ in $\mathbb{C}_z$. On the other hand, a Kähler manifold with Kähler potential $F$  also completely determines the geometry of the Sasakian manifold $M$. We refer to the metric $h$ as the underlying Kähler metric of the Sasakian manifold. The Ricci form of the underlying Kähler metric is given by $2\I \operatorname{R}d z \wedge d\bar{z}$ with
\begin{align}
    \operatorname{R} &= -\partial_{z\bar{z}}\log \left(F_{z\bar{z}}\right),\label{ricKah}
\end{align}

From Lemma \ref{lemmaframcoframe}, for the embedding \eqref{embed} the expressions for the frame \eqref{coofr} simplify to
\begin{align}\label{sasframe}
\begin{split}
    \partial &= \partial_z+\I F_z\p_u, \qquad \partial_o = 2F_{z\bar{z}}\partial_u
    \end{split}   
\end{align}
and the function $c$ takes the form
\begin{align} \label{fun-c} 
    c &= -\p_z\log\left (F_{z\bar{z}}\right ).
\end{align}

We now show that $3$-dimensional Sasakian manifolds can also be characterised through the structure function $c$ defined by \eqref{cdef}.
\begin{prop}\label{sasaki}
Let $(M,D,J)$ be a 3-dimensional strictly pseudoconvex CR manifold. Suppose it admits a non-constant CR function $z$. Let $(\mu,\lambda)$ be an adapted coframe for $M$, where $\mu = dz \neq 0$, and let $(\partial,\partial_o)$ be the dual frame. Let $c$ be the structure function defined by \eqref{cdef}. Then $M$ is Sasakian if and only if 
$$\partial_{o}c = 0.$$
\end{prop}
\begin{proof}
If $(M,D,J)$ is Sasakian then with respect to the coordinates $(z,u + \I v)$ from the embedding \eqref{embed} 
$$\partial_o = \frac{1}{\psi}\partial_u.$$ 
Since the function $c$ from \eqref{fun-c} does not depend on $u$, it follows that
$$\partial_{o}c = 0.$$
Proving now the converse, we assume $\partial_{o}c = 0$. We show that there exists a real function $A$, such that the vector field
$$\cZ = \operatorname{e}^\varphi \partial_o,$$
is an infinitesimal CR automorphism transversal to $D$, i.e. it is the Reeb vector field for the contact form
$$\operatorname{e}^{-\varphi}\lambda,$$
and it preserves the complex structure on the CR distribution, i.e.
$$\bar{\partial} \, \lrcorner \, \mathcal L_{\cZ} \mu = \bar{\partial} \, \lrcorner \, (d \left(\cZ \, \lrcorner \, \mu \right) + \cZ \, \lrcorner \, d\mu) = 0.$$ 
The condition
$$\operatorname{e}^{\varphi}\partial_o \, \lrcorner \, d \left(\operatorname{e}^{-\varphi}\lambda\right)=0,$$
is equivalent to
\begin{equation}\label{eq1}
    (\partial \varphi - c)\mu + (\bar{\partial} \varphi - \bar{c})\bar{\mu} = 0.
\end{equation}
Therefore, it remains to show that the equation 
\begin{equation}\label{eq2}
    \partial \varphi = c,
\end{equation}
has a real solution. Let $(z = x + \I y, u)$ be local coordinates as above, such that $\partial_o = \frac{1}{\psi}\partial_u$. Since $\partial_{o} c = 0$, it follows from \eqref{cid} that
$$\partial_z \bar{c} = \partial_{\bar{z}}c.$$ 
Substituting  $c(x,y) = a(x,y) + \I b(x,y)$ and  $\partial_z = \frac{1}{2}(\partial_x - \I \partial_y)$ into the above equation gives us
\begin{equation}\label{equa111}
    b_x = -a_y.
\end{equation}
For the real function $\varphi$, the equation \eqref{eq2} is equivalent to
\begin{equation}\label{equa222}
    \begin{cases}
        \varphi_x = 2a   \\
        \varphi_y = -2b.  
    \end{cases}
\end{equation}
Now, the condition \eqref{equa111} guarantees the existence of a local solution $\varphi(x,y)$ with $\partial_{o}\varphi = 0$. 
It follows, 
\begin{equation*}
    \partial \varphi = \partial_z \varphi - L\partial_u \varphi = \partial_z \varphi = c. \qedhere
\end{equation*}
\end{proof}
As a special case we consider the so-called tubular hypersurfaces $v = F(y)$, which have an underlying Kähler potential $F(y)$ depending only on one real coordinate. In this case, from the Lemma \ref{lemmaframcoframe}, the expressions for the frame \eqref{coofr} and structure function simplify to
\begin{equation}
    \partial = \partial_z +\frac{1}{2} F_y \partial_u, \qquad \partial_o = \frac{1}{2} F_{yy}\partial_u, \qquad c = \frac{\I F_{yyy}}{2F_{yy}}. \label{tubeframe}
\end{equation}

\section{Shearfree congruences of null geodesics and CR geometry}
In this section, we recall some notions related to shearfree congruences of null geodesics on Lorentzian manifolds following the paper by Hill, Lewandowski and Nurowski \cite{HLN}.
\begin{definition}\label{shearfree}\cite{AGS}
A shearfree congruence of null geodesics on a 4-dimensional Lorentzian manifold $(\mathscr{M},g)$ is a foliation by integral curves of a nowhere vanishing vector field $\ks$, such that
	\begin{itemize}
		\item[(i)] The vector field $\ks$ is null, i.e. $g(\ks,\ks) = 0$.
		
		\item[(ii)] $\mathscr{L}_{\ks} g = \rho g + \theta \chi $, where $\theta = g(\ks,\cdot)$, $\rho$ is a real function on $\mathscr{M}$ and $\chi$ is a 1-form.
	\end{itemize}
For the sake of brevity we will call the vector field $\ks$ shearfree for the Lorentzian metric $g$.	
\end{definition}
In the definition above, and throughout this paper, we use the standard notation of the symmetric tensor product of two 1-forms $\omega$ and $\gamma$,
$$\omega \gamma = \frac{1}{2}\left(\omega \otimes {\gamma} + \gamma \otimes \omega \right).$$
Notice that conditions (i) and (ii) in the definition above mean that the metric $g$ changes conformally under the flow of $\ks$, when restricted to the distribution
$${\ks}^\perp = \{X \in T\mathscr{M} \, | \  g(X,\ks) = 0\}.$$
In this sense a shearfree vector field can be considered as a generalisation of a conformal Killing vector field.

Following \cite{HLN}, we define, for a given 3-dimensional CR manifold $(M,[(\mu, \lambda)])$, a class of Lorentzian metrics on the line bundle $\mathscr{M} = M \times \mathbb{R}$
\begin{equation}\label{shearfreemetric}
    g = 2{\cP}^2\left(\mu \bar{\mu} + \lambda(dr + \cW\mu + \overline{\cW}\bar{\mu} + \cH \lambda)\right).
\end{equation}
Here $\cP$ is a non-zero real function, $r,\cH$ are real functions and $\cW$ is a complex function on $\mathscr{M}$. We have used the notation $\lambda$ and $\mu$ for the forms on $M$ as well as for their pullbacks to $\mathscr M$. We will call $(\mathscr M,g)$ a \emph{lift} of the CR manifold $M$.

The vector field $\ks = \frac{\partial}{\partial r}$ is shearfree for all metrics of the family \eqref{shearfreemetric}. This family of metrics is independent of the choice of coframe $(\mu, \lambda),$ depending only on the CR structure of $M$. With respect to the choice of coframe on $\mathscr{M}$
\begin{align}\label{coframe}
    {\theta}^1 = \cP\mu, \qquad {\theta}^2 = \cP\bar{\mu}, \qquad {\theta}^3 = \cP\lambda, \qquad {\theta}^4 = \cP \left(dr + \cW\mu + \overline{\cW}\bar{\mu} + \cH\lambda \right),
\end{align}
and corresponding frame
\begin{equation}\label{frame}
    e_1 = \frac{1}{\cP}(\partial - \cW{\partial}_r), \quad e_2 = \frac{1}{\cP}(\bar{\partial} - \overline{\cW}{\partial}_r), \quad e_3 = \frac{1}{\cP}(\partial_o - \cH{\partial}_r), \quad e_4 = \frac{1}{\cP}{\partial}_r,
\end{equation}
the metric \eqref{shearfreemetric} takes the form
\begin{equation}\label{metric2}
    g = 2(\theta^1\theta^2 + \theta^3\theta^4).
\end{equation}
In the Theorem \ref{theoHLN}  below, proved in \cite{HLN}, we summarise a list of consequences for the functions $\cP, \cW$ and $\cH$ resulting from the vanishing of the complexified Ricci curvature restricted to the so-called $\alpha$-planes,  that is the following components with respect to the frame \eqref{frame} chosen above 
$$\Ric_{44}=\Ric_{22}=\Ric_{24} = 0.$$

Then the Goldberg-Sachs theorem \cites{GoldSa, GHN} implies that the following components of the Weyl tensor vanish
$$\operatorname{C_{4141}}=\operatorname{C_{4341}}=0.$$
Here, again, the subscripts indicate the components of the tensor with respect to the frame \eqref{frame}.
Together with 
$$\operatorname{C_{4132}}\neq 0$$
this establishes that the Lorentzian metric is of Petrov type II or D.

\begin{theo}\cite{HLN}\label{theoHLN}
1. Assume that for the metric \eqref{shearfreemetric} the components of the Ricci curvature $\Ric_{44}$, $\Ric_{22}$, $\Ric_{24}$ with respect to the frame \eqref{frame} vanish. Then
then 
\begin{equation}\label{eqx}
    \cW = \I \mathcal{X} e^{-\I r} + \mathcal{Y},
\end{equation} 
where $\mathcal{X}$ and $\mathcal{Y}$ are complex-valued functions, which do not depend on $r$ and the following statements hold:
            \begin{equation}\label{R44}
               \cP = \frac{p}{\cos(\frac{r + s}{2})},
            \end{equation}
            
            \begin{align}
               \p t + (c - t)t = 0, \quad t = c + 2\p \log p - e^{\I s} \mathcal{X}, \label{equ-t}
            \end{align}
              
            \begin{equation}\label{eqx-y}
                \mathcal{Y} = \I c + 2 \I \p\, \log p + \p s - 2\I t,  
            \end{equation}
where $p$ and $s$ are real-valued functions with $p_r = s_r = 0$, and the complex functions $c$, $\mathcal{X}$, and $\mathcal{Y}$ are defined by \eqref{2.1} and \eqref{eqx}, respectively.

2. Vice versa, if a metric  \eqref{shearfreemetric} satisfies \eqref{eqx},  \eqref{R44}, \eqref{equ-t}, \eqref{eqx-y} then 
$$\Ric_{44}=\Ric_{22}=\Ric_{24}=0.$$

3. Furthermore, assuming $\Ric_{44}=\Ric_{22}=\Ric_{24}=0$,

            \begin{equation}\label{equa-p-s}
                \Ric_{12} = \Ric_{34} = \Lambda \iff \mathcal{B} = 0,
            \end{equation}
with            
        
            \begin{small}        
                \begin{equation*}
                    \mathcal{B} = \left(\p\bar{\p} + \bar{\p}\p + \bar{c}\,\p + c\,\bar{\p} + \frac{1}{2}|c|^2 + \frac{3}{4}\left (\bar{\p}c + \p \bar{c}\right) - \frac{3}{2}\left(\p \bar{t} + \bar{\p}t + |t|^2\right)\right)p - \frac{m + \bar{m}}{p^3} - \frac{2}{3}\Lambda p^3,
                \end{equation*}
            \end{small}

where $m$ is a complex function satisfying $m_{r} = 0$.

4. Assuming   $\Ric_{44}=\Ric_{22}=\Ric_{24}=0$ and  $\Ric_{12} = \Ric_{34} = \Lambda$, we have
            \begin{equation}
                \Ric_{13} = 0 \iff \p m + 3(c - t)m = 0. \label{R13}
            \end{equation}

5. Finally, assuming $\Ric_{44}=\Ric_{22}=\Ric_{24}=0$,  $\Ric_{12} = \Ric_{34} = \Lambda$  and  $\Ric_{13} = 0$, we have
             \begin{equation}\label{pet2}
                \operatorname{C_{4132}} = \frac{(1 + e^{\I (r + s)})^3}{2p^6}m.
            \end{equation}

\end{theo}
We will use the results of the theorem above in the case of Sasakian manifolds, given by a defining equation \eqref{embed} in coordinates $(z,w = u + \I v)$. In this setting, the identity  \eqref{cid}, $\p\bar{c} = \bar{\p} c$, allows us to simplify some of the expressions. Moreover, one may notice that $t = 0$ is a solution of the equation \eqref{equ-t}. For the remainder of the paper, we proceed under the ansatz $t \equiv 0$, with the CR structure on $M$ given by the adapted frame and coframe defined in \eqref{sasframe}. In addition, we assume $s = 0$ and $p_u = 0$. As a consequence of these assumptions and Theorem \ref{theoHLN}, the function $\cH$ takes the form
\begin{equation}\label{curlH}
    \cH = \frac{m}{p^4}e^{2\I r} + \frac{\bar{m}}{p^4}e^{-2\I r} + \cQ e^{\I r} + \bar{\cQ}e^{-\I r} + \cT,
\end{equation}
where
\begin{align}
    &\cQ = \frac{3m + \bar{m}}{p^4} + \frac{2}{3}\Lambda p^2 + \frac{2\p p\, \bar{\p} p - p \left (\p \bar{\p} p + \bar{\p}\p p \right)}{2p^2} - \bar{\p} c\label{funQ}\\
    &\cT = \frac{3m + 3\bar{m}}{p^4} + 2\Lambda p^2 + \frac{2\p p\, \bar{\p} p - p \left (\p \bar{\p} p + \bar{\p}\p p \right)}{p^2} - 2\bar{\p} c,\label{funI}
\end{align}
and
\begin{small}
    \begin{equation}\label{R33}
        \Ric_{33} = \Bigg(\frac{8}{p^4} \left (\partial + 2c\right )\Big [p^2 \left (\partial \bar{\cI} - 2 \Lambda \left( 2\bar{\p} \log p + \bar{c}\right)p^2 \right)\Big] + \frac{16\Lambda}{p}\mathcal{B}_o + \frac{16\I}{p^3}\p_o\left(\frac{m}{p^4} \right)\Bigg) \cos^4\left(\frac{r}{2}\right),
    \end{equation}
\end{small}
where the function $\cI$ is defined by
\begin{equation}\label{eqI}
    \cI = \p\left(\p\log p + c \right) + \left(\p\log p + c \right)^2,
\end{equation}
and
\begin{equation}\label{Bo}
    \mathcal{B}_o = \mathcal{B}|_{t = 0} = \left(\p\bar{\p} + \bar{\p}\p + \bar{c}\,\p + c\,\bar{\p} + \frac{1}{2}|c|^2 + \frac{3}{4}\left (\bar{\p}c + \p \bar{c}\right )\right)p - \frac{m + \bar{m}}{p^3} - \frac{2}{3}\Lambda p^3.
\end{equation}
\section{Quasi-Einstein Lorentzian manifolds}
Hill, Nurowski and Lewandowski proved in \cite{HLN} that for any real analytic CR manifold there exists a representative in the family of metrics \eqref{shearfreemetric} of Petrov type II or D, which is quasi-Einstein, that is, it satisfies the equations
$$\operatorname{Ric}(g) = \Lambda g + \Phi\lambda^2,$$
where $\Phi$ is some real function. We prove an analogous result for Sasakian manifolds without the assumption of real analyticity. 
 
Prior to presenting our main result, we state a theorem proved by Du and Ma in \cite{DM}, (see also \cite{DD} by Dong and Du) which gives sufficient conditions for the existence of a solution for the so-called logistic elliptic equation. The existence of such solution is crucial in the proof of our theorem.\\
\begin{theo}\cites{DD, DM}\label{logeq}
Consider the logistic elliptic equation
	\begin{equation}\label{logic}
		-\Delta q = a(x)q - b(x)q^{\sigma}, \ x \in \mathbb{R}^{n},
	\end{equation}
where, $\sigma > 1$, and $a(x), b(x)$ are continuous real-valued functions with $b(x) > 0$, for all $x \in \mathbb{R}^{n}$. Then, the equation \eqref{logic} has a unique positive solution $q$, if the following limits exist and are positive
	\begin{equation*}
	    \displaystyle{\lim_{|x| \to \infty}}a(x) > 0 \quad \text{and}\quad \displaystyle{\lim_{|x| \to \infty}}b(x) > 0.
	\end{equation*}
\end{theo}
In the following subsections, we prove Theorems \ref{sasaki-theoint}, \ref{tubular-theo}, \ref{ESAKint}.
\subsection{Proof of Theorem \ref{sasaki-theoint} }
Suppose that $\left(M,D,J\right)$ is a 3-dimensional Sasakian manifold given by an equation $v = F(z,\bar{z})$,  such that the limits
    \begin{equation}\label{limits}
        \displaystyle{\lim_{|z| \to \infty}}F_{z\bar{z}} > 0 \quad \text{and} \quad \displaystyle{\lim_{|z| \to \infty}}\partial_{z\bar{z}}\log \left(F_{z\bar{z}}\right) < 0
    \end{equation}
exist, and let $\Lambda$ be an arbitrary positive constant. 

Consider the metric \eqref{shearfreemetric} with 
\begin{align*}
\mu&=dz,\\
\lambda&=\frac{1}{2F_{z\bar{z}}}\left (du-\I F_z+\I F_{\bar{z}}\right ).
\end{align*}
We determine the functions $\cP, \mathcal{X}, \mathcal{Y}$ and $\cH$ using Theorem \ref{theoHLN}. From \eqref{equ-t} and \eqref{eqx-y} the functions $\mathcal{X}, \mathcal{Y}$ and $\cP$ are
\begin{align*}
\mathcal{X}=c + 2\p \log p, \qquad \mathcal{Y}=\I \left ( c + 2\p \log p\right ), \qquad  \cP = \frac{p}{\cos(\frac{r }{2})},
\end{align*}
where  $c$ is given by \eqref{fun-c}. We now determine the function $p$. Note that 
 the equation $\mathcal{B}=0$ from \eqref{equa-p-s} with our ansatz $t=0$ becomes $\mathcal{B}_o=0$, where $\mathcal{B}_o$ is given by \eqref{Bo}. Assuming that $m$ is purely imaginary and that $p$ does not depend on $u$ this simplifies further to
\begin{equation}\label{equ-p-s1}
    p_{z\bar{z}} + \frac{\bar{c}}{2} \,p_z + \frac{c}{2}\,p_{\bar{z}} + \left(\frac{|c|^2}{4}+ \frac{3}{4} \p_z\bar{c}\right)p = \frac{\Lambda}{3} p^3.
\end{equation}
Let $q$ be the function defined by
\begin{equation}\label{subs1}
	p = f\,q, \quad \text{where} \quad f = \sqrt{F_{z\bar{z}}}.
\end{equation}
In terms of $q$ the equation \eqref{equ-p-s1} becomes the logistic elliptic equation
\begin{equation}\label{equ-q-s}
    q_{z\bar{z}} + a(z,\bar{z})\,q = b(z,\bar{z})\, q^3,
\end{equation}

with
\begin{equation}\label{equ-q-a}
    a(z,\bar{z}) = \frac{f_{z\bar{z}}}{f} + \frac{f_z}{2f}\bar{c} + \frac{f_{\bar{z}}}{2f}c + \frac{|c|^2}{4} + \frac{3}{4}\p_z\bar{c}
\end{equation}
and
$$b(z,\bar{z}) = \frac{\Lambda}{3} f^2 = \frac{\Lambda}{3}F_{z\bar{z}}.$$
The expression  \eqref{equ-q-a}  for the function $a(z,\bar{z})$ simplifies to 
$$a(z,\bar{z}) = -\frac{1}{4}\partial_{z\bar{z}}\log \left (F_{z\bar{z}}\right) = \frac{1}{4}\operatorname{R},$$
where $ \operatorname{R}$ is given by \eqref{ricKah}.
By our assumptions the limits \eqref{limits} of both $a$ and $b$ exist and are positive real numbers, and, by Theorem \ref{logeq}, the logistic elliptic equation \eqref{equ-q-s} has a unique positive solution. This determines the function $p$.

It remains to choose the purely imaginary function $m$ such that $R_{13}=0$ and $\operatorname{C_{4132}}\neq 0$. Let $m = \I(\p_o\bar{w})^3$, where
$w = u + \I F(z, \bar{z})$ is a CR function, that is $\bar{\p}w = \left(\p_{\bar{z}} - \I F_{\bar{z}}\p_u\right)w = 0.$ From \eqref{sasframe} it follows that
$$\p_o w  = 2F_{z\bar{z}} \neq 0,$$
is a real function, hence $m$ is purely imaginary. Now the non-zero function $m$ satisfies
$$\p m + 3cm = 3\I (\p_o \bar{w})^2\p\p_o \bar{w} + 3\I c(\p_o \bar{w})^3 = 3\I (\p_o \bar{w})^2(\p_o\p \bar{w} - c\p_o\bar{w}) + 3\I c(\p_o \bar{w})^3 = 0,$$
since $\p\bar{w} = 0$ and $\p_o\p - \p\p_o = c\p_o$.  Then, by \eqref{R13}, $\Ric_{13} = 0$ and the function $\cH$ defined by \eqref{curlH} is determined. Since, $m\ne 0$, the Weyl component 
$$\operatorname{C_{4132}}=\frac{(1 + e^{\I r })^3}{2p^6}m$$
is non-zero, and therefore, 
 the metric $g$ is of Petrov type II or D.
\hfill $\Box$ \medskip

In the special case when $M$ is a tubular manifold, the logistic equation reduces to an ODE and the conclusion of Theorem  \ref{sasaki-theoint} holds for any real constant $\Lambda$, without the limit conditions \eqref{limits} on $a,b$.
\subsection{Proof of Theorem \ref{tubular-theo}}
Let $\left(M,D, J\right)$ be a 3-dimensional tubular CR manifold with defining function $v = F(y)$. The proof is the same as the proof of Theorem  \ref{sasaki-theoint}  with one exception. The equation  $\mathcal B_o=0$ becomes the ODE
 \begin{equation}\label{equ-tp1}
    p_{yy} - \I \bar{c}\,p_y + \I c\,p_y + \left({|c|}^2 - \frac{3}{2}\I\, \bar{c}_y\right)p = \frac{16}{3}\Lambda p^3,
\end{equation}
which, by the Picard–Lindelöf theorem has a two-parametric family of solutions for any $\Lambda$. We remark that by the same substitution as in the proof of Theorem \ref{sasaki-theoint} the ODE \ref{equ-tp1} simplifies to 
\begin{equation}\label{equ-tq2}
    q_{yy} + R\,q = \frac{16}{3}\Lambda f^2 q^3.
\end{equation}
\phantom{x}\hfill $\Box$ \medskip
%
\subsection{Proof of Theorem \ref{ESAKint}}
Let $(M,D,J)$ be a 3-dimensional Sasakian manifold with defining equation $v = F(z,\bar{z})$, and assume that the underlying Kähler manifold with Kähler metric
$$h = 2F_{z\bar{z}} dz \, d\bar{z},$$
 is Einstein with non-zero Einstein constant $\Lambda_o$ i.e., 
 $$\Ric(h) = R = \Lambda_o F_{z\bar{z}}.$$
 We follow again the proof of Theorem \ref{sasaki-theoint}. The equation \eqref{equ-q-s} takes now the form
\begin{equation}\label{logeq1}
    q_{z\bar{z}} + \left (\frac{\Lambda_o}{4}F_{z\bar{z}} \right)q = \left (\frac{\Lambda}{3}F_{z\bar{z}}\right ) q^3.
\end{equation}
For any real constant $\Lambda$, such that $\Lambda\Lambda_o > 0$ the constant function
$$q = \pm \sqrt{\frac{3\Lambda_o}{4\Lambda}},$$
is a solution of the above equation \eqref{logeq1}. 

It remains to show that in this case $R_{33}=0$, hence the metric \eqref{shearfreemetric} becomes Einstein.
Notice that the function $\cI$ from \eqref{eqI} now becomes 
$$\cI = \frac{1}{2}\p c + \frac{1}{4}c^2,$$ 
and consequently,
$$\p_z\bar{\cI} = \frac{1}{2}\p_{z} \left(\p_{\bar{z}} \bar{c}\right) + \frac{1}{2}\bar{c}\,\p_z \bar{c} = \frac{\Lambda_o}{2}\p_{\bar{z}} F_{z\bar{z}} - \frac{\Lambda_o}{2}\p_{\bar{z}} F_{z\bar{z}} = 0.$$
This reduces \eqref{R33} to
\begin{equation*}
	  \Ric_{33} = \frac{8\cos^4\left(\frac{r}{2}\right)}{p^4} \left(\partial + 2c\right) \left(\frac{16\I}{p^3}\p_o\left(\frac{m}{p^4} \right) \right). 
\end{equation*}
Since
$$\p_o m=8\I \p_o\left (F_{z\bar{z}}\right )=0, $$
and $\p_o p = 0$, this implies $\Ric_{33} = 0.$    \hfill $\Box$

\section{Examples}

Let $(M,D,J)$ be a 3-dimensional Sasakian manifold with defining function
\begin{equation}\label{harmonic}
    v = F(z,\bar{z}) = z\overline{\varphi(z)} + \bar{z}\varphi(z),
\end{equation}
where $\varphi$ is a holomorphic function, so that
$$c = -\frac{\varphi''(z)}{\overline{\varphi(z)} + \varphi'(z)}.$$
Then, there exists a 4-dimensional lift with Ricci flat Lorentzian metric $g$ of Petrov type II or D. The metric $g$ is given by
$$g = \frac{2\Re \varphi'}{\cos^2(\frac{r}{2})}\left(h + \lambda'\left(dr + \cW dz + \overline{\cW}d\bar{z} + \frac{\cH}{4\Re \varphi'} \lambda' \right) \right),$$
where
$$h = 4\Re \varphi'\, dz \, d\bar{z}, \quad \lambda' = \psi\lambda, \quad \cW = -\I c (e^{-\I r} + 1),$$
and the function $\cH$ is defined by \eqref{curlH} with $m = \mathrm{i}\psi^{-3}.$

Note that for the defining function \eqref{harmonic}
$$\left(F_{z\bar{z}}\right)_{z\bar{z}} = 0.$$
Replacing this condition by
$$\left(\left(F_{z\bar{z}}\right)^{\frac{2}{3}}\right)_{z\bar{z}} = 0,$$
yields another Ricci flat 4-dimensional lift of Petrov type II or D. The choices
$$p = \left(F_{z\bar{z}}\right)^{\frac{2}{3}}, \quad m = \mathrm{i}\psi^{-3},$$
together with $\Lambda = 0$, solve all differential equations of Theorem \ref{theoHLN}. The resulting metric becomes
\begin{equation}\label{HFRT}
    g = \frac{\left (F_{z\bar{z}}\right)^{\frac{4}{3}}}{\cos^2(\frac{r}{2})}\left(h + \lambda'\left(dr + \cW dz + \overline{\cW}d\bar{z} + \frac{\cH}{2F_{z\bar{z}}} \lambda' \right) \right),
\end{equation}
with 
$$\cW = -\frac{\I}{3}c(e^{-\I r} + 1),$$
and the function $\cH$ defined by \eqref{curlH} with $m = \mathrm{i}\psi^{-3}.$ The metric \eqref{HFRT} is an example of the Fefferman-Robinson-Trautman metric introduced in \cite{SG} in relation with the embedding problem of a 3-dimensional CR manifold. For more details see \cite{SG}.

\end{document}